\documentclass[12pt]{amsart}
\usepackage{amsmath,amsfonts,amsthm,amscd,amssymb,stmaryrd,mathrsfs}
\usepackage{graphicx,color}

\newtheorem{theorem}{Theorem}
\newcommand{\beq}{\begin{eqnarray}}
\newcommand{\eeq}{\end{eqnarray}}

\newcommand{\R}{\mathbb{R}}

\newcommand{\tr}{\mathrm{tr}}

\newcommand{\PP}{\mathcal{P}}

\newcommand{\ric}{\mathrm{Ric}}

\newcommand{\II}{\mathscr{I}}
\newcommand{\ii}{{\bf I}}
\newcommand{\ff}{\mathscr{G}}
\newcommand{\mP}{\mathcal{P}}
\newcommand{\EE}{\mathscr{E}}
\newcommand{\vol}{\mathrm{vol}}
\newtheorem{thmx}{Theorem}
\newcommand{\LL}{\mathrm{II}}

\newtheorem{proposition}[theorem]{Proposition}
\newtheorem{lemma}[theorem]{Lemma}

\newtheorem{corollary}[theorem]{Corollary}

\newtheorem{remark}[theorem]{Remark}

\newtheorem{convention}[theorem]{Convention}

\numberwithin{equation}{section}
\numberwithin{theorem}{section}
\usepackage{amsmath}
\usepackage{amssymb}
\usepackage{latexsym}

\title{Sobolev-trace inequalities of order four}
\author{Antonio G. Ache}
\address{Mathematics Department, Princeton University, Fine Hall, Washington Road,
Princeton New Jersey 08544-1000 USA }
\email{aache@math.princeton.edu}
\author{Sun-Yung Alice Chang}
\address{Mathematics Department, Princeton University, Fine Hall, Washington Road,
Princeton New Jersey 08544-1000 USA }
\email{chang@math.princeton.edu}
\thanks{The first author was supported by a postdoctoral fellowship of the
National Science Foundation, award No. DMS-1204742.}
\thanks{The second author was supported by NSF grant DMS-1104536}
\begin{document}

\maketitle
\bibliographystyle{amsalpha}

\begin{abstract} We establish sharp trace Sobolev inequalities of order four on Euclidean $d$-balls for $d\ge 4$. When $d=4$, our inequality generalizes the classical second order Lebedev-Milin inequality on Euclidean $2$-balls. Our method relies on the use of scattering theory on hyperbolic $d$-balls. As an application, we characterize extremals of the main term  in the log-determinant formula corresponding to the conformal Laplacian coupled with the boundary Robin operator on Euclidean $4$-balls.  
\end{abstract}

\section{Introduction} In this paper we derive some sharp Sobolev trace
inequalities on Euclidean $d$-balls of order 4 with the $(d-1)$-sphere as boundary. 
Sobolev inequalities for critical exponents play an important role in 
problems in conformal geometry.  One example is the sharp Sobolev inequality of order
2 on the standard sphere $(S^n, g_{S^n})$ for $ n>2$ which takes the form
\begin{align} \label{sob-n}
\left(\frac{1}{\vol(S^{n})}\oint_{S^n}|f|^{q}d\sigma\right)^{2/q}\le\frac{q-2}{n\cdot\vol(S^{n})}\oint_{S^{n}}|\tilde{\nabla}f|^2d\sigma+\frac{1}{\vol(S^{n})}\oint_{S^{n}}|f|^2d\sigma,
\end{align}
where $d\sigma$ is the Lebesgue measure induced by $g_{S^{n}}$, $\vol(S^{n})$ is the volume of $S^{n}$ with respect to $g_{S^{n}}$, $\tilde{\nabla}$ is the gradient 
operator on $(S^n,g_{S^n})$ and $q=\frac{2n}{n-2}$.   
Inequality \eqref{sob-n} has been crucial in the study of the Yamabe problem on closed manifolds. We remark that when $n=2$, \eqref{sob-n} takes the form of
Moser-Onfori inequality
\begin{align}\label{onofri}
\log\left( \frac{1}{4\pi}\oint_{S^{2}}e^{f}d\sigma\right)\le\frac{1}{4\pi}\oint_{S^2}fd\sigma+\frac{1}{16\pi}\oint_{S^2}|\tilde{\nabla}f|^2d\sigma,
\end{align}
which is fundamental in the problem of prescribing Gaussian curvature on
the sphere $S^2$ (see (\cite{moser}, \cite{onofri}, see also \cite{ops3,ops2,ops1})). 
On manifolds with boundary, a class of inequalities analogous to
 \eqref{sob-n} and \eqref{onofri}  is the class of 
\emph{Sobolev trace inequalties}. A classical example of these 
inequalities is the following: consider the unit ball $B^d$
on Euclidean space $\mathbb{R}^d$ with the sphere $S^{d-1}$ as boundary and let 
$f$ be a function in $C^{\infty}(S^{d-1})$. 
If $v$ is a smooth extension of $f$ to $B^d$  
we have the inequality    
\begin{align}\label{sobtrace-d} 
\frac{d-2}{2}\vol(S^{d-1})^{\frac{1}{d-1}}\left(\oint_{S^{d-1}}|f|^{\frac{2(d-1)}{(d-2)}}d\sigma\right)^{\frac{d-2}{d-1}}&\le \int_{B^{d}}|\nabla v|^2dx+\frac{d-2}{2}\oint_{S^{d-1}}f^{2}d\sigma\\
&\text{when}~d>2.\nonumber
\end{align}
When $d=2$, inequality \eqref{sobtrace-d} becomes the classical Lebedev-Milin inequality \cite{LM}
\begin{align}\label{lm} 
\log\left(\frac{1}{\pi}\oint_{S^1}e^{f}d\sigma\right)\le\frac{1}{4\pi}\int_{D}|\nabla v|^2dx+\frac{1}{\pi}\oint_{S^1}fd\sigma,
\end{align}
where now $\nabla v$ is the gradient of $v$ with respect to the Euclidean metric on the disk $D=\{x\in\R^{2}:|x|\le 1\}$.
Inequality \eqref{sobtrace-d} has been derived by Escobar \cite{esc} and applied 
to study the Yamabe problem on manifolds with boundary; while the Lebedev-Milin inequality \eqref{lm}
has been applied to a wide variety of problems in classical analysis, including
the Bieberbach conjecture {\cite{dBr}) and by Osgood-Phillips-Sarnak (\cite{ops3,ops2,ops1})
 in the study
of the compactness of isospectral planar domains. 

In this note, we consider extensions of some of the above results to order 4. 
The order here refers to the order of operator involved in the derivation of
the inequalities. For the inequalities cited above, the operators which describe
the Euler-Lagrange  equation of the extremal functions involved in the proof of the 
 inequality  are either Laplace operator or the conformal Laplace operator, both of order 2. 
 In recent years, the role played by these second order operators has been 
vastly extended to operators of order 4 (e.g. the bi-Laplace operator or the Paneitz operator of order 4, 
which we will briefly discuss in section \ref{adapted}). 
For example, the Sobolev inequality which prescribes the embedding of
$ W_{o}^{2,2} $  to $L^{\frac {2n}{n-4}}$ on $\R^n$ for $n >4$ has a fourth order
analogue on $S^{n}$, namely  
\begin{align}\label{equ; sob4-on-sn}
C_n \left(\frac{1}{\vol(S^{n})}\oint_{S^n} |f|^{\frac {2n}{n-4}}d\sigma\right)^{\frac {n-4}{n}} \le \frac{1}{\vol(S^n)} \oint_{S^n}f P_4fd\sigma
\end{align} 
where $C_n = \frac {n (n-2)(n+2)(n-4)}{ 2^4}$, and $P_4$ is the 4-th order Paneitz operator on $S^n$ defined as 
\begin{align*}
P_4 = \left(-\tilde{\Delta}+\left(\frac{n}{2}\right)\left(\frac{n-2}{2}\right)\right)\left(-\tilde{\Delta}+\left(\frac{n+2}{2}\right)\left(\frac{n-4}{2}\right)\right),
\end{align*}
(see \cite[Theorem 1.1]{branson95}) where $\tilde{\Delta}$ is used to denote the Laplacian with respect to the metric $g_{S^{n}}$ (see \cite[Theorem 6]{be93}. When $n=4$, the
 inequality becomes an inequality of Moser-Onofri  type just as in \eqref{onofri}, and
takes the form
\begin{align}\label{equ: sob4-s4}
\log\left(\frac{1}{\vol(S^{4})} \oint_{S^4} e^{4 ( w - \bar{w})}d\sigma\right) \le\frac {1}{3\vol(S^4)} \oint_{S^4}((\tilde{\Delta} w)^2+2 |\tilde{\nabla} w|^2 )d\sigma,
\end{align}
which is a special case of Beckner's log-Sobolev inequality, see \cite{be93}, 
also see Chang (\cite{chang96}, Lecture 5).
As mentioned above, the main results in this paper are extensions of inequalities 
\eqref{sobtrace-d} when $d>2$  and \eqref{lm} above
to sharp Sobolev trace inequalities of order 4. We will adopt the following notational convention
\begin{convention}\em{ We will sometimes use $g_{0}$ to denote the Euclidean metric on $\R^{d}$. The Euclidean metric 
is also commonly denoted as $|dx|^2$.  
}
\end{convention}
We start by stating our result for dimension $d>4$ 

\begin{thmx}\label{dge4} 
Given $f\in C^{\infty}(S^{d-1})$ with $d>4$, suppose $v$ is a smooth extension of $f$ to the unit ball $B^{d}$ which also satisfies the Neumann boundary condition
\begin{align}
\frac{\partial}{\partial n}v|_{\partial B^d}=-\frac{(d-4)}{4}f. \label{dNeumann}
\end{align}
 Then we have the inequality
\begin{align}\label{mainineqge4}
c_{d}\left(\vol(S^{d-1})\right)^{\frac{3}{d-1}}&\left(\oint_{S^{d-1}}|f|^{\frac{2(d-1)}{d-4}}d\sigma\right)^{\frac{d-4}{d-1}}\nonumber\\
&\le\int_{B^d}(\Delta_{g_0} v )^2dx
+ 2 \oint_{S^{d-1}}|\tilde{\nabla}f|^2d\sigma+  b_{d}\oint_{S^{d-1}}|f|^2d\sigma,   
\end{align}
where $c_d=\frac{d(d-2)(d-4)}{4}$, $b_{d}=\frac{d(d-4)}{2}$ and $\tilde{\nabla}f$ is the gradient of $f$ with respect to the round metric $g_{S^{d-1}}$. 
Moreover, equality holds if and only if $v$ is a bi-harmonic extension of a function of the form
$f_{z_{0}}(\xi)=c|1-\langle z_{0},\xi\rangle|^{\frac{4-d}{4}}$ where $c$ is a constant, $\xi\in S^{d-1}$, $z_{0}$ is some point in $B^{d}$ and $v$ satisfies the boundary condition \eqref{dNeumann}. When $f\equiv 1$, equality in \eqref{mainineqge4} is attained by the function $v(x)=1+\frac{d-4}{4}(1-|x|^2)$.  
 \end{thmx}

The corresponding theorem for dimension $4$ is the following inequality which is a generalization of the classical Lebedev-Milin inequality \eqref{lm}:

\begin{thmx}\label{deq4}
Given $\varphi\in C^{\infty}(S^{3})$, suppose $w$ is a smooth extension of $\varphi$ to the unit ball $B^{4}$. 
If $w$  satisfies the Neumann boundary condition
\begin{align}
\frac{\partial}{\partial n}w|_{\partial B^4}=0,\label{regneumann}
\end{align}
then we have the inequality
\begin{align}\label{mainineq4}
\log\left(\frac{1}{2\pi^2}\oint_{S^3}e^{3(\varphi-\bar{\varphi})}d\sigma\right)\le \frac{3}{16\pi^2}\int_{B^4}(\Delta_{g_0} w)^2dx+\frac{3}{8\pi^2}\oint_{S^{3}}|\tilde{\nabla}\varphi|^2 d\sigma,
\end{align}
Moreover, equality holds if and only if $w$ is a bi-harmonic extension of a function of the form  
$\varphi_{z_{0}}(\xi)=-\log\left|1-\langle z_0,\xi\rangle\right|+c$ where $c$ is a constant, $\xi\in S^{3}$, $z_{0}$ is some point in $B^4$ and $v$ satisfies the boundary condition \eqref{regneumann}. 
\end{thmx}
We remark that the method that we used to discover and establish the
above sharp inequalities is non-traditional in the sense that we first
derived such inequalities for some metric $g^{*}$ on $B^d$ which is in the conformal
class of the Euclidean metric $ g_{0}=|dx|^2$, and then we derived the  desired inequalities in Theorems \ref{dge4} and \ref{deq4} through
conformal covariant properties between the 4th order 
``Paneitz operator" on $g^{*}$ and the bi-Laplace operator on the 
flat metric $g_{0}=|dx|^2$. The choice of the metric $g^{*}$ in turn comes from
the consideration of the hyperbolic metric $g_+$ on $B^d$ and the connection
between $g^{*}$ and $g_+$ is established through the scattering theory. 
In section 2 below we will explain this connection and why the choice of $g^{*}$ 
is a ``natural" one for the Sobolev inequalities in Theorem \ref{dge4} 
and \ref{deq4} above. We remark that when we apply the same procedure to the 
\emph{second order} conformal Laplacian on $B^d$, the $g^{*}$ metric we obtain 
agrees with $g_{0}$. In this sense, our choice of the metric $g^{*}$ is natural. 

We also point out that once the sharp inqualities in the statements of Theorem A and B 
are formulated, it is not surprising that one can find more direct proofs of these two theorems, 
but we would like to record here our original approach to discover such
sharp inequalities.  We also believe that our method may lead to the formulation 
of other sharp inequalities.

\subsection{Application to Polyakov-\'{A}lvarez type formulas} We conclude by providing an application to the study of the extremals of the functional determinants 
for conformally covariant operators on manifolds of dimension four with boundary. 

First we recall the classical Polyakov-\'{A}lvarez formula for surfaces with 
boundary: if $(M,\partial M,g)$ is a surface with boundary and we let 
$g_{w}$ denote the metric 
\begin{align*}
g_{w}=e^{2w}g,
\end{align*}
where $w$ is a smooth function on $M$, then there is a formula 
relating the functional determinants associated to the Laplace Beltrami 
operators $\Delta_{g_{w}}$ and $\Delta_{g}$ and this formula is written 
entirely in terms of integrals of local invariants of $(M,g)$ and of 
$\partial M$. More precisely, we have 
\begin{align}
F[w]:=\log\left(\frac{\det(\Delta_{g_w})}{\det(\Delta_{g})}\right)&=-\frac{1}{6\pi}\left(\frac{1}{2}\int_{M}|\nabla w|_{g}^2dV_{g}+\int_{M}wK_{g}dV_{g}+\oint_{\partial M}wk_{g}ds_{g}\right)\nonumber\\
&-\frac{1}{4\pi}\left(\oint_{\partial M}k_{g_{w}}(ds)_{g_{w}}-\oint_{\partial M}k_{g}(ds)_{g}\right).\label{polalv}
\end{align} 
Here $K_{g}$ represents the Gauss curvature of $g$, $k_{g}$ is the geodesic curvature of $\partial M$ with respect to $g$ and $(ds)_{g}$ is the line element on $\partial M$ with respect to $g$. In the celebrated work of Osgood-Phillips-Sarnak \cite{ops3,ops2,ops1} it is shown that there is a relationship between the Lebedev-Milin inequality \eqref{lm} and the study of extremals of the ratio $w\rightarrow F[w]$.  More precisely, if in \eqref{polalv} we isolate the functional determinant coefficient
\begin{align}\label{F1}
F_{1}[w]=-\frac{1}{6\pi}\left(\frac{1}{2}\int_{M}|\nabla w|_{g}^2dV_{g}+\int_{M}wK_{g}dV_{g}+\oint_{\partial M}wk_{g}ds_{g}\right),
\end{align} 
then $F_{1}[w]$ is conformally invariant and it is proved in \cite{ops2} 
that inequality \eqref{lm} can be applied to show that $F_{1}[w]$ has a definite sign when the background metric is the Euclidean metric.

A fundamental point in the derivation of \eqref{polalv} is that the Laplace Beltrami operator $\Delta_{g}$ has a conformal invariance property in dimension $2$, 
however, this conformal property fails in higher dimensions.  A result extending \eqref{polalv} for functional determinants of conformally covariant operators 
on compact manifolds with boundary in dimension $4$ was obtained by Chang and Qing in \cite{cqI} based on an earlier preliminary version of the formula by Branson and Gilkey \cite{BG}. In \cite{cqI}, Chang and Qing considered the \emph{Conformal Laplacian operator} $L_{g}$ given explicitly in dimension $4$ by 
\begin{align*}
L_{g}=-\Delta_{g}+\frac{1}{6}R_{g}, 
\end{align*} 
and they coupled $L_{g}$ with a boundary operator, namely the \emph{Robin operator} defined as
\begin{align*}
\mathfrak {B}_g=\frac{\partial}{\partial n}+\frac{1}{3}H,
\end{align*}   
where $n$ is the outward normal to the boundary $\partial M$ and $H$ is the mean curvature of $\partial M$ with respect to $g$. 
The pair $(L_{g}, {\mathfrak B}_{g})$ and its conformal properties are fundamental 
in the study of the Yamabe problem with boundary 
(see for example \cite{esc}). 
Let $(M^{4},\partial M,g)$ be a compact manifold with boundary and 
let $g_{w}=e^{2w}g$ where $w$ is a smooth function defined on $\overline{M}$. 
For the pair, $(L_{g},{\mathfrak B}_g)$, it is proved in \cite{cqI} that there is a Polyakov-\'{A}lvarez formula of the form 
 \begin{align}\label{11terms}
 \log\left( \frac {\det (L_{g_w},\mathfrak{B}_{g_w})}
{\det (L_g, \mathfrak{B}_g)}\right)=\sum_{i =1}^{11} \gamma_{i} {\mathbf I}_{i} [g, w],
\end{align}
where the terms ${\mathbf I}_{i} [g, w]$ for $i=1,\ldots,11$ are computed in terms of integrals of local invariants of $M$ and $\partial M$ with respect to the metrics $g$ and $g_{w}$.
Moreover the second term ${\mathbf I}_{2} [g, w]$ is conformally invariant and plays the same role as $F_{1}[w]$ in \eqref{polalv} (where $F_{1}[w]$ is given by \eqref{F1}). 
We will show an explicit expression for ${\mathbf I}_{2} [g, w]$ in Section 6 (see equation \eqref{I2formula}).  This is why we refer to ${\bf I}_{2}[g,w]$ as 
the \emph{main term in the functional determinant expansion} \eqref{11terms}. 
In (\cite{cqI}, Corollary 3.4), under two additional geometric assumptions, one can show that ${\mathbf I}_{2} [g, w]$ has a definite sign on the model case $(B^4,S^3,g_0)$.

As an application of the sharp inequality in Theorem \ref{deq4}, we will here show that the metric $g^*$ is a minimizer of the functional $\ii_{2}$  over the class of functions with prescribed boundary value and Neumann boundary condition.   
More precisely, letting $C_{\varphi}$  and $\tilde{C}_{\varphi}$ be the following classes of functions
\begin{align}
C_{\varphi}=\left\{w\in C^{\infty}(B^4):w|_{S^3}=\varphi,\frac{\partial}{\partial n}w|_{S^3}=0\right\}, \label{defC}
\end{align}
and
\begin{align}
\tilde{C}_{\varphi}=\left\{\chi\in C^{\infty}(B^4):\chi|_{S^3}=\varphi,\frac{\partial}{\partial n}w|_{S^3}=-1\right\}\label{delCtilde}.
\end{align} 
Observe that any function $\chi$  in $\tilde{C}_{\varphi}$ can be written as $\chi=w+\rho$ where $w\in C_{\varphi}$. We will prove that the main term in the expansion \eqref{11terms} is non-negative for the metric $g^*$ and we also describe in detail the equality case: 
\begin{thmx}\label{thmFD}
Given $\varphi \in C^{\infty}(S^3) $ with $\int_{S^3} e^{3\varphi} d \sigma = |S^3|$, 
then for all $w\in C_{\varphi}$, we have the following inequalities
\begin{enumerate}
\item $\ii_{2}[g^{*},w]\ge 0$ with equality if and only if $\varphi(\xi)=-\log\left|1-\langle z_{0},\xi\rangle\right|+c$ where $z_{0}\in B^4$ is fixed, $\xi\in S^3$, $c$ is a constant and $w$ is a bi-harmonic extension of $\varphi$ in $C_{\varphi}$. It follows that $g^*$ is a global minimizer of $\ii_{2}$ for functions in the class $C_{\varphi}$
\item $\ii_{2}[g_{0},\chi]\ge 0$ where $\chi\in\tilde{C}_{\varphi}$. Equality is attained for $\varphi=-\log\left|1-\langle z_{0},\xi\rangle\right|+c$ where $z_{0}\in B^4$ is fixed, $\xi\in S^{3}$, $c$ is a constant and $\chi$ is a bi-harmonic extension of $\varphi$ in $\tilde{C}_{\varphi}$.   
\end{enumerate}
\end{thmx}
\begin{remark} \em{In \cite{cqI,cqII} an inequality similar to the one stated in Theorem \ref{deq4} was proved under very restrictive conditions. See for example \cite[Lemma 3.4]{cqII}} 
\end{remark}

In a forthcoming joint work of the authors together with Jeffrey Case, we will 
give some further applications of the lower bound in Theorem \ref{thmFD} 
and also extend the study to extremals for the log-determinant functional \eqref{11terms}. 

\subsection{Acknowledgements} 
The authors would like to thank Jeffrey Case for many helpful discussions.

\section{Adapted Metrics}\label{adapted}
As we have mentioned in the introduction, one of the key steps used in the proof
of the inequalities in this paper is the consideration of a special metric called
``adapted metric" which 
we will denote by $g^{*}$, which lies in the conformal class of the flat metric
$g_0 = |dx|^2$ on the Euclidean unit ball $B^{d}$ and is with totally 
geodesic boundary. The consideration of the metric $g^{*}$, introduced in the
work of Case-Chang \cite{CC14}, is natural from conformal geometry point of
view, and via the connnection of the metric as a preferred conformal compactification
(as we will explain below) of the hyperbolic metric on $B^d$. 

We now recall the definition of the $g^{*}$ metric in a general setting and
will derive its specific form in the special setting of $B^{d}$ with the sphere as
boundary, and when the corresponding boundary operator is of order three.
\subsection{GJMS operators, scattering theory}
We first recall the definition of the GJMS operators via scattering theory \cite{GZ03}.  A triple $(X^{n+1},M^n,g_+)$ is a \emph{Poincar\'{e}-Einstein manifold} if
\begin{enumerate}
\item $X^{n+1}$ is (diffeomorphic to) the interior of a compact manifold with boundary $\partial X =M^n$,
\item $(X^{n+1},g_+)$ is complete with $\ric(g_+)=-ng_+$, and
\item there exists a nonnegative $r \in C^\infty(X)$ such that $r^{-1}(0)=M^n$, $dr\ne 0$ along $M$, and the metric $g:=r^2g_+$ extends to a smooth metric on $X^{n+1}$.
\end{enumerate}
A function $r$ satisfying these properties is called a \emph{defining function}, since $r$ is only determined up to multiplication by a positive smooth function on $X$, it is clear that only the conformal class $[h]:=[g|_{TM}]$ on $M$ is well-defined for a Poincar\'e--Einstein manifold.

Given a Poincar\'e--Einstein manifold $(X^{n+1},M^n,g_+)$ and a representative $h$ of the conformal boundary, there exists a unique defining function $r$, called the \emph{geodesic defining function}, such that, locally near $M$, the metric $g_+$ takes the form $g_+ = r^{-2}\left(dr^2 + h_r\right)$
for $h_r$ a one-parameter family of metrics on $M$ with $h_0=h$ and having an asymptotic expansion involving only even powers of $r$, at least up to order $n$ (cf. \cite{FG12, GL91,L95}).

It is well-known (see \cite{GZ03,MM87} for more general statements) that given $f\in C^\infty(M)$ and $s\in C$ such that $ \mathrm{Re}(s)>\frac{n}{2}$, $s\not\in\frac{n}{2}+ \mathbb{N}_{0}$, where $\mathbb{N}_{0}=\mathbb{N}\cup\{0\}$, and $s(n-s)$ is not in the pure-point spectrum $\sigma_{\mathrm{pp}}(-\Delta_{g_+})$ of $-\Delta_{g_+}$, the Poisson equation
\begin{align}
\label{eqn:poisson_equation}
-\Delta_{g_+}u - s(n-s)u = 0~\text{in $X$},
\end{align}
has a unique solution of the form
\begin{equation}\label{eqn:poisson_asymptotics}
u = Fr^{n-s} + Hr^s \qquad\text{for $F,H\in C^\infty\left(X\right)$ and $F |_M=f$}.
\end{equation}
Indeed, $F$ has an asymptotic expansion
\begin{equation}
\label{eqn:F_expansion}
F = f_{(0)} + f_{(2)}r^2 + f_{(4)}r^4 + \dotsb
\end{equation}
for $f_{(0)}=f$ and all the functions $f_{(2\ell)}$ are determined by $f$.  The \emph{Poisson operator} ${\mP}(s)$ is the operator which maps $f$ to the solution $u={\mP}(s)f$, and this operator is analytic for $s(n-s)\not\in\sigma_{\mathrm{pp}}(-\Delta_{g_+})$.  The \emph{scattering operator} is defined by $S(s)f = H|_M$.  This defines a meromorphic family of pseudodifferential operators in $\mathrm{Re}(s)>\frac{n}{2}$.  
Given $\gamma \in (0, \frac {n}{2})$, Graham and Zworski (\cite {GZ03}) defined
the \emph{fractional GJMS operator} $P_{2 \gamma}$ as the operator 
\begin{equation}
\label{eqn:scattering_definition}
P_{2\gamma}f := d_\gamma S\left(\frac{n}{2}+\gamma\right)f \qquad\text{for $d_\gamma=2^{2\gamma}\frac{\Gamma(\gamma)}{\Gamma(-\gamma)}$}.
\end{equation}
For $\gamma\in\mathbb{N}_0$, this definition recovers the GJMS operators \cite{GJMS1992}. For $\gamma\in\left(0,\frac{n}{2}\right)$, Graham and Zworski showed that $P_{2\gamma}$ is a formally self-adjoint pseudodifferential operator with principle symbol equal to the principle symbol of $(-\Delta)^\gamma$, and moreover, if $\hat h=e^{2\tau}h$ is another choice of conformal representative of the conformal boundary, then
\begin{equation}
\label{eqn conformal covariant}
 \hat P_{2\gamma}f = e^{-\frac{n+2\gamma}{2}\tau} P_{2\gamma}\left(e^{\frac{n-2\gamma}{2}\tau}f\right), 
\end{equation}
for all $f\in C^\infty(M)$.  Together these properties justify the terminology ``fractional GJMS operator.''

We adopt the convention that for $\gamma\in\left(0,\frac{n}{2}\right)$, the \emph{fractional $Q$-curvature $Q_{2\gamma}$} is the scalar
\begin{equation}
\label{eqn:q_scattering}
Q_{2\gamma} := \frac{2}{n-2\gamma}P_{2\gamma}(1) .
\end{equation}
In particular, we emphasize that this definition produces a well-defined invariant in the critical case $2\gamma=n$, with the corresponding $Q$ curvature 
a generalization of Branson's Q-curvature (in the case when $ 2 \gamma = n =4$.
see \cite {br93},\cite{GZ03}).\\

In this paper we will only deal with the special case when $ \gamma = \frac {3}{2}$ and the corresponding GJMS operator $P_{3}$ of order 3. We will also treat
$P_3$ as the``boundary operator" of the correponding GJMS operator. 
In the case $2\gamma =4$, the GJMS operator $P_{4}$ is also called 
\emph{Paneitz operator} and is an elliptic operator of  
order four defined on the manifold $X^{d}$ (see \cite{Paneitz}) given by
\begin{equation}
\label{P4}
(P_4)_{g}  = (-\Delta_g) ^2 \, +\, \delta_{g} \left((4 A_{g} -(d-2) J_{g} g)(\nabla\cdot,\cdot)\right) \,+ \, \frac {d-4}{2} (Q_4)_{g},
\end{equation} 
where $\delta$ denotes divergence, $\nabla$ denotes gradient on functions,  $A_{g}$ the Schouten tensor $A = \frac {1}{d-2} (\ric -J_{g} g)$,  $J_{g} =
\frac {1}{2(d-1)} R_{g}$, $R_{g}$ is the scalar curvature of the metric $g$ and $Q_{4}$ is the $Q$-curvature 
\begin{equation}
\label{Q4}
(Q_4)_g = - \Delta_g J_g + \frac {d}{4} J_g^2 - 2 |A_g|_g^2.
\end{equation}
We remark that in general, properties and formulas 
for $P_{2k}$ and $Q_{2k}$ have been recursively derived (\cite{Juhl}, \cite{FG13}).
The conformal invariance property of $(P_4,Q_4)$, which is a special
case of \eqref{eqn conformal covariant} is best expressed as
\begin{equation}
\label{eqn conf4d} 
(P_4)_{\hat g}  U = (\psi)^{ \frac{d+4}{d-4}} \,(P_4)_g ( \psi U),
\end{equation}
for all smooth function $U$ defined on manifolds of dimension $ d \geq 5$, where $ \hat g \in [g]$ denotes
the metric $ \hat g = (\psi)^{ \frac {4}{d-4}} g$. 
and 
\begin{equation}
\label{eqn conf4}
(P_4)_{\hat g}  U = e^{ - 4 \tau} (P_4)_g (U),
\end{equation}
for all smooth function $U$ defined on manifolds of dimension $ d =4$, where 
$ \hat g \in [g]$ denotes
the metric $ \hat g = e^{2\tau} g $. 
\subsection{Adapted metrics when  $\gamma = \frac{3}{2}$}
In the earlier paper of Case-Chang \cite{CC14}, we have introduced a special
metric called ``adapted metric" to compactify a given conformally compact
Poincare Einstein manifold $(X^{n+1}, \partial X, g_{+})$. The metric was introduced for any parameter
$ s = \frac {2}{n} + \gamma$ for $ \gamma \in (0, \frac {n}{2})$ and for
$ s=n$ when $n$ is an odd integer.
For such an $s$, we consider the Poisson equation \eqref{eqn:poisson_equation}
with Dirichlet data $ f \equiv 1$, and denote the solution of the equation by
$v_s$. By a result of J. Lee \cite{L95}, if one assumes further that the Yamabe
constant of the boundary metric $h$ is positive (this happens when the scalar
curvature of some metric in $[h]$ is positive), then there is no point 
specturm of $ - \Delta_{g_{+}}$ and hence $v_s >0 $ on $X$. Thus we can take
$y_s := (v_s)^{\frac {1}{n-s} }$ as a defining function, and $g_s = y_s^2 g_{+}$
is defined as the ``adapted metric". In the limiting case, when  $ s = n$ and
$n$ is an odd integer, i.e. when $ \gamma = \frac{n}{2}$, the adapted metric
as appeared in the literature in the work of Fefferman-Graham (\cite{FG02}, see 
also \cite{CQY06}) and is defined as follows:
\begin{align}
\tau = - \frac{d}{ds}|_{s=n} v_s, 
\end{align}
then $\tau$ satisfies 
\begin{align}\label{eqw}
- \Delta_{g_{+}}\tau = n,
\end{align}
and the adapted metric is defined as $ g^{*} = e^{2\tau} g_{+}$.\\

\indent The adapted metric for all $s$ satisfies many good curvature
properties stated in terms of  ``smooth metric with measures". Here we will 
just cite the properties when $s= \frac n2 + \frac 32$ satisfied by the metric $g^* := g_{ \frac {n}{2} + \frac{3}{2} }$ which we will use in the rest of this paper.    
\begin{proposition}[Properties of $g^*$]\label{properties} See Lemma 6.2 and Lemma 7.6 in \cite{CC14}. 
\begin{enumerate}
\item[(i)] It is a metric with totally geodesic boundary.
\item[(ii)] $Q_4(g^{*}) \equiv 0.$
\item[(iii)]We define the energy of $U$ with respect to the Paneitz operator $(P_{4})_{g}$ as the integral quantity obtained by 
dropping the boundary terms when integrating by parts of the integral $\int_{X^{d}}\left\{(P_{4})_{g^*}U\right\}UdV_{g^*}$, i.e.,
\begin{align}\label{paneitzenergy}
E_{4}(g^*)[U]=\int_{X^d}(\Delta_{g^{*}} U)^2 -(4 A_{g^{*}} - (d-2) J_{g^{*}} g^{*}) \langle \nabla_{g^*} U, \nabla_{g^*} U\rangle  dV_{g^*}
\end{align}
For all smooth function $f$ defined on $\partial X^{d}$, we have the identity 
\begin{align}
\label{eqn e4}
 \frac {1}{2} E_4(g^{*})[U_f] = \int_{\partial X} P_3 f f d\sigma_{g^{*}}-\frac{d-4}{2} \int_{\partial X} Q_3 (g^{*}) f^2d\sigma_{g^{*}}.
\end{align}
where $U_f$ denotes the (unique) solution of the equation:
\begin{align}
(P_4)_{g^{*}} U_f= 0&~\text{on}~ X^d,\label{p40} \\
U_f = f &~\text{on} ~\partial X,\label{uff}\\
\frac {\partial U_f} {\partial n_{g^{*}}} = 0&~ \text{on}~ \partial X.\label{duf}
\end{align}
\end{enumerate}
\end{proposition}

\subsection{Explicit formula for $g^*$ metric on $(B^d, S^{d-1}, g_H)$}\label{explicit} In this section we derive explicit formulas for the metric $g^{*}$ in the 
model case of $(B^d, S^{d-1}, g_H) $, which happens to be computable. 
\begin{proposition}\label{proppsi} On the model case $(B^d, S^{d-1}, g_H) $, i.e., for $g_{H}=\rho^{-2}|dx|^2$, $ \rho = \frac{1-|x|^2}{2}  $, we have
\begin{enumerate}
\item[(i)] When $d \geq 5$, $g^{*} = (\psi)^{ \frac {4}{d-4}} |dx|^2$, where
\begin{align*}
 \psi = 1 + \frac{d-4}{2} \rho.
 \end{align*}
\item[(ii)] When $d =4$, $ g^{*} = e^ {2 \rho} |dx|^2 = e^{ (1- |x|^2) } |dx|^2$.
\end{enumerate}
\end{proposition}
\begin{remark}\em{
The metric $g^*$ in (ii)  is called \emph{Fefferman-Graham metric} and has been described in several previous papers
including \cite{FG02,CQY06,CC14}. The Fefferman Graham metric is constructed via scattering theory and satisfies several important 
properties, including 
\begin{itemize}
\item $S^{d-1}$ is totally geodesic under $g^*$
\item The $Q$-curvature of $g^{*}$ is zero
\item $R_{g^*}>0$. 
\end{itemize}
}
\end{remark}
\begin{proof}
(i) Denote $ d = n+1$, and $ s = \frac{n}{2} + \frac {3}{2} $, by our definition,
 $ g^{*} = {v_s^{\frac {2}{n-s}} } g_H $, where $g_H$ is the hyperbolic metric on
$B^d$ and $v_s$ is the unique solution of the Poisson equation with Dirichlet
data $ f \equiv 1$. 
\begin{equation}
\label{poisson_equationB}
-\Delta_{g_H}v_s - s(n-s)v_s = 0 \qquad\text{in $B$}
\end{equation}
We now recall $g_H = \rho^{-2} g_0$, where $\rho = \frac{ 1-|x|^2}{2}$ and
$g_0 = |dx|^2$ the flat metric on $B^d$. Using this notation, we find that
$ g^{*} = \psi^{\frac{4}{d-4}} g_0$ and where $ \psi = v_s \rho^ {- \frac{d-4}{2}}
= v_s \rho ^{s-n}$. One of the main observations in the paper of Chang-Gonzalez
(\cite {CG11} ) is that in the compactified metric $g_0 = \rho^2 g_H$, equation
\eqref{poisson_equationB} is equivalent to the equation
\begin{align}\label{chang_gonzalez}
-\delta_{g_0} (\rho^{a}\nabla \psi)+\EE(\rho) \psi=0,
\end{align}
where $ a = 1 - 2 \gamma = -2$ in our case. The error term $\EE(\rho)$ is given by
\begin{align}
\EE(\rho)=-\rho^{a/2}\left(\Delta_{g_0}\rho^{a/2}\right)+\left(\gamma^{2}-\frac{1}{4}\right)\rho^{a-2}+\frac{n-1}{4n}R_{g_0}\rho^{a},
\end{align}
which simplifies to 
\begin{align*}
\EE(\rho)=-\Delta_{g_{0}}(\rho^{-1})\rho^{-1}+2\rho^{-4}.
\end{align*}
Note that  
\begin{align*}
\frac{\partial}{\partial r}\rho=-r,\\
\frac{\partial^{2}}{\partial r^{2}}\rho=-1. 
\end{align*}
Observe also that
\begin{align*}
\Delta_{g_{0}}\rho^{-1}=\frac{\partial^{2}}{\partial r^{2}}\rho^{-1}+(d-1)r^{-1}\frac{\partial}{\partial r}\rho^{-1},
\end{align*}
and 
\begin{align*}
\frac{\partial}{\partial r}\rho^{-1}&=r\rho^{-2},\\
\frac{\partial^{2}}{\partial r^{2}}\rho^{-1}&=\rho^{-2}+2r\rho^{-3}, 
\end{align*}
so that $\Delta_{g_{0}}\rho^{-1}=\frac{2}{\rho^{3}}+\frac{d-4}{\rho^{2}}$ and 
\begin{align*}
\EE(\rho)=-\rho^{-1}\Delta_{g_{0}}(\rho^{-1})+2\rho^{-4}=-\frac{d-4}{\rho^{3}}. 
\end{align*}
Observe that in this case, it is reasonable to expect that $\psi$ will be a radial function. Assuming so, we can compute the term 
\begin{align*}
-\delta_{g_{0}}(\rho^{-2}\nabla \psi)&=-2r\rho^{-3}\psi^{'}+\rho^{-2}\Delta_{g_{0}}\psi\\
&=-2r\rho^{-3}\psi^{'}-\rho^{-2}\left(\psi^{''}+\frac{d-1}{r}\psi^{'}\right)\\
&=\frac{4}{r}\rho^{-3}\left(\rho-\frac{1}{2}\right)\psi^{'}-\rho^{-2}\left(\psi^{''}+\frac{d-1}{r}\psi^{'}\right)\\
&=-\rho^{-2}\psi^{''}-\frac{d-5}{r}\rho^{-2}\psi^{'}-\frac{2}{r}\rho^{-3}\psi^{'}
\end{align*}
and equation \eqref{chang_gonzalez} is equivalent to 
\begin{align}\label{psiode}
-\rho^{-2}\psi^{''}-\frac{d-5}{r}\rho^{-2}\psi^{'}-\frac{2}{r}\rho^{-3}\psi^{'}-(d-4)\rho^{-3}\psi=0.
\end{align}
We can now verify that $\psi = 1 + \frac{d-4}{2} \rho$ is a solution of the ODE \eqref{psiode}
with the desired boundary condition $\psi|_{\partial B} \equiv 1$. This finishes the proof of (i).  

We now prove (ii). When $d=4$, one can either follow a procedure similar as (i) to compute 
explicitly the solution of the PDE $- \Delta_{g_H} \tau = d-1 = 3 $ on $B^4$. Or denote 
$r=|x|$, we find that a radial solution of $- \Delta_{g_H} \tau = d-1 = 3$ 
satisfies the equation 
\begin{align*}
\frac{\partial^{2}}{\partial r^{2}}\tau(r)+\frac{6r}{r^{2}-4}\frac{\partial}{\partial r}\tau(r)-\frac{2}{r}\frac{\partial}{\partial r}\tau(r)=-\frac{3}{r^{2}},
\end{align*}
and a solution to this equation satisfying the boundary condition $e^{2\tau}g_{H}|_{S^{3}}=g_{S^{3}}$ is given by 
\begin{align*}
\tau(x)=\log\left(\frac{1-|x|^2}{2}\right)+\frac{1-|x|^2}{2}. 
\end{align*}
Our adapted metric will then be $g^{*}=e^{2\tau}\rho^{-2} g_{0}=e^{(1-|x|^{2})}g_{0}$. 
We remark alternatively, we can find the metric $g^{*}$ in dimension $d=4$ by a ``dimension continuity" argument of Branson, by computing the limit  
\begin{align*}
e^{2\tau} \rho^{-2} =\lim_{d \rightarrow 4} (\psi_d)^{ \frac {4}{d-4} } = e^ {2 \rho}.
\end{align*}
\end{proof}

\begin{remark}\em{ As pointed out in the introduction, if in \eqref{p40} we replace the Paneitz $P_{4}$ operator by the conformal Laplacian and carry out the construction in Proposition \ref{proppsi} , we obtain $g^{*}=g_{0}$. More precisely, if we consider the system
\begin{align}
L_{g^*}U_{f}&=0,~\text{in}~B^{d}\label{Leq}\\
U_{f}&=f,~\text{on}~S^{d-1},\label{Uff2}\\
\frac{\partial}{\partial n_{g^*}}U_{f}&=0~\text{on}~S^{d-1}\label{dUf2},
\end{align}
where $L_{g^*}=-\Delta_{g^*}+\frac{(d-2)R}{4(d-1)}$, then
\begin{enumerate}
\item[(i)] when $d\ge 3$, $g* = (\psi)^{2/(d-2)}g_{0}$, then $\psi \equiv 1$ by solving the system \eqref{Leq}-\eqref{dUf2} at $s = n/2 + 1/2$ for the Dirichlet data $f \equiv 1$). 
\item[(ii)]when $d=2$, by solving \eqref{eqw} for $n=1$ on $\mathbb{H}^2$, then $g* = e^{2\tau} g_H = g_{0}$.
\end{enumerate}
}
\end{remark}

\section{Energy identity for the model case in dimension $d>4$}\label{energydge4}
We now write the energy identity in Proposition \ref{properties} for the metric $g^*$ in the model case $(B^{d},S^{d-1},g_{H})$ for $d>4$ using the results of section \ref{explicit}. Recall that in this case the adapted metric is given by 
$g^{*}=\psi^{\frac{4}{d-4}}g_{0}$ where $\psi=1 + \frac{d-4}{2}\rho$ and $\rho(x)=\frac{1-|x|^2}{2}$. Observe that from the definition of the energy $E(g)$ in Proposition \ref{properties} as the functional obtained by dropping the boundary terms when integrating by parts in $\int_{X^{d}}\left\{(P_{4})_{g}u\right\}udV_{g}$ we have 
\begin{align}\label{explicitE}
E_{4}(g)[u]=\int_{X^{d}}u(P_{4})_{g}u dV_{g}+\II_{1}(g)[u]-\II_{2}(g)[u]-\II_{3}(g)[u],
\end{align}
where $\II_{1}$, $\II_{2}$ and $\II_{3}$ defined for general $u$ by 
\begin{align}
\II_{1}(g)[u]&=\oint_{M^{d-1}}\Delta_{g}u\frac{\partial}{\partial n_{g}}u d\sigma_{g},\label{I1}\\
\II_{2}(g)[u]&=\oint_{M^{d-1}}u\frac{\partial}{\partial n_{g}}(\Delta_{g}u)d\sigma_{g},\label{I2}\\
\II_{3}(g)[u]&=\oint_{M^{d-1}}uh(n_{g},\nabla u)d\sigma_{g},\label{I3}
\end{align}
with $h=4\left(A_{g}-(d-2)Jg\right)$. In view of \eqref{explicitE}, let us write the energy identity \eqref{eqn e4} in Proposition \ref{properties} as 
\begin{align}
\int_{B^{4}}U_{f}(P_{4})_{g^*} U_{f}dV_{g^*}&+\II_{1}(g^{*})[U_{f}]-\II_{2}(g^{*})[U_{f}]+\II_{3}(g^{*})[U_{f}]\label{E1}\\
&=2\oint_{S^{d-1}}P_{3}fd\sigma-(d-4)\oint_{S^{d-1}}Q_{3}f^{2}d\sigma\label{E2},
\end{align}
where $U_{f}$ satisfies \eqref{p40}-\eqref{duf}. For the following three lemmas, we will consider functions $u$ (which includes the case when $u=U_f$),  which satisfies the following two conditions
\begin{align}
u|_{\partial B^4}&=f,\label{udir}\\
\frac{\partial}{\partial n_{g^{*}}}u&=0.\label{uneumann} 
\end{align}
Note that on $S^{d-1}$ we have $n_{g_{0}}=n_{g^{*}}$, thus the Neumann boundary condition \eqref{uneumann} can also be written in terms of $g_{0}$.
We start by noting the following lemma that holds for all dimensions $d\ge 4$. 
\begin{lemma}\label{I3lemma}
Let $d\ge 4$. For the metric $g^{*}$ and $u$ satisfying \eqref{udir} and \eqref{uneumann} we have $\II_{3}(g^{*})[u]=0$ for any function $f\in C^{\infty}(S^{d-1})$.  
\end{lemma}
\begin{proof}
Recall that $\II_{3}(g^{*})[ u]=\displaystyle{\oint_{S^{d-1}}fh(n_{g^{*}},\nabla  u)d\sigma}$, where $h=4A_{g^*}-(d-2)J_{g^{*}}g^{*}$ and by virtue of \eqref{duf} we have $g^{*}(n_{g^{*}},\nabla  u)=0$, so that
\begin{align*}
h(n_{g^{*}},\nabla  u)=\frac{4}{d-2}\ric_{g^{*}}(n_{g^*},\nabla  u). 
\end{align*}
On the other hand, if for $d\ge 4$ we write $g^{*}=e^{2\tau}g_{0}$ where of course
\begin{align*}
\tau(x)=\left\{
\begin{array}{cc}
\rho(x)&~\text{for}~d=4,\\
\frac{2}{d-4}\log\left(1+\frac{d-4}{2}\rho(x)\right)&~\text{for}~d\ge 5,
\end{array}
\right.
\end{align*}
we have by the transformation law for the Ricci tensor the identity 
\begin{align*}
\ric_{g^{*}}=(d-2)\left(-\nabla^{2}_{g_{0}}\tau-\frac{1}{(d-2)}(\Delta_{g_{0}}\tau)g_{0}+d\tau\otimes d\tau-|\nabla \tau|^2g_{0}\right) 
\end{align*}
and since $\tau$ is radial we have 
\begin{align*}
\ric_{g^{*}}=\theta dn\otimes dn+\omega g_{0}, 
\end{align*}
for radial functions $\theta$ and $\omega$ and then clearly $\ric_{g^{*}}(n_{g^{*}},\nabla  u)|_{\partial B^4}=0$. This implies $\II_{3}(g^{*})[ u]=0$ as needed.  
\end{proof}
We now relate the boundary terms $\II_{1}(g^{*})[ u]$ and $\II_{2}(g^{*})[ u]$ to boundary integrals computed in terms of the Euclidean metric $g_{0}$. To be more precise, recall that 
the operator  $(P_{4})_{g}$ satisfies the following conformal covariance property: if we write $g^{*}=\psi^{\frac{4}{d-4}}g_{0}$ then 
  \begin{align*}
(P_{4})_{g^{*}}(u)=\psi^{-\frac{d+4}{d-4}}P_{g_{0}}(\psi u),
\end{align*}
and therefore
\begin{align*}
 &\int_{B^{d}}u(P_{4})_{g^*}(u)dV_{g^{*}}\\
 &=\int_{B^4}u\psi^{-\frac{d+4}{d-4}}(P_{4})_{g_{0}}(\psi u)\psi^{\frac{2d}{d-4}}dV_{g_{0}}\\
 &=\int_{B^{d}}\psi u(P_{4})_{g_{0}}(\psi u)dV_{g_{0}}\\
 &=\oint_{S^{d-1}}f\frac{\partial}{\partial n_{g_{0}}}(\Delta_{g_{0}}(u\psi))-\frac{\partial}{\partial n_{g_{0}}}(u\psi)\Delta_{g_{0}}(u\psi)d\sigma+\int_{B^{4}}(\Delta_{g_{0}}\psi u)^{2}dV_{g_{0}},\\
 &=\II_{1}(g_{0})[\psi u]-\II_{2}(g_{0})[\psi u]+\int_{B^{4}}(\Delta_{g_{0}}\psi u)^{2}dV_{g_{0}},
\end{align*}
and since $\displaystyle{\int_{B^{4}}(\Delta_{g_0}\psi u)^{2}dV_{g_0}}$ is the Paneitz energy of $\psi u$ with respect to $g_{0}$ , we are interested in computing the difference
\begin{align*}
\II_{1}(g^*)[u]-\II_{2}(g^{*})[u]-\left(\II_{1}(g_{0})[\psi u]-\II_{2}(g_{0})[\psi u]\right).
\end{align*}
We will do this not for general $u$ but for $u= U_{f}$ satisfying \eqref{p40}-\eqref{duf}. We start with the following lemma
\begin{lemma}
For $ u$ satisfying \eqref{p40}-\eqref{duf} we have the identity
\begin{align*}
\II_{1}(g^{*})[ u]-\II_{1}(g_{0})[\psi  u]&=\frac{d-4}{2}\oint_{S^{d-1}}f\left(\frac{\partial^{2}}{\partial n^2_{g_{0}}} u+\tilde{\Delta} f\right)d\sigma
\\
&-d\frac{(d-4)^{2}}{4}\oint_{S^{d-1}} f^{2}d\sigma. 
\end{align*}
\end{lemma}
\begin{proof}
On $S^{d-1}$ we have $n_{g^{*}}=n_{g_{0}}$ and $n_{g_{0}}$ can be identified with the normal direction $\frac{\partial}{\partial r}$. Since in addition $\II_{1}(g^{*})[ u]=0$, we start by noting
\begin{align*}
\II_{1}(g^{*}, u)-\II_{1}(g_{0})[\psi  u]=-\II_{1}(g_{0})[\psi u]=-\oint_{S^{d-1}}\frac{\partial}{\partial n_{g_{0}}}(\psi  u)\Delta_{g_{0}}(\psi  u)d\sigma
\end{align*}
and 
\begin{align*}
&=\frac{\partial}{\partial n_{g_{0}}}(\psi  u)\Delta_{g_{0}}( u\psi)|_{\partial B^4}\\
&=f\psi^{'}(1)\psi(1)\left(\frac{\partial^{2}}{\partial r^{2}} u|_{r=1}+\tilde{\Delta} u\right)+ u^{2}\psi^{'}(\Delta_{g_{0}}\psi)|_{r=1},\\
&=-\frac{d-4}{2} u\left(\frac{\partial^{2}}{\partial r^{2}} u|_{r=1}+\tilde{\Delta} u\right)+d\frac{(d-4)^{2}}{4} u^{2},\label{diffIg}\\
&=-\frac{d-4}{2}f\left(\frac{\partial^{2}}{\partial n_{g_{0}}^{2}} u+\tilde{\Delta}f\right)+d\frac{(d-4)^{2}}{4}f^{2}
\end{align*}
from where the result follows.  
\end{proof}
We also have 
\begin{lemma}
For any $u$ satisfying \eqref{udir} and \eqref{uneumann}, the following identity holds 
\begin{align*}
 \II_{2}(g^{*})[ u]-\II_{2}(g_{0})[\psi  u]=\frac{d}{2}\oint_{S^{d-1}}f\tilde{\Delta}fd\sigma+\frac{d-4}{2}\oint_{S^{d-1}}f\frac{\partial^{2}}{\partial n_{g_{0}}^{2}} ud\sigma.
\end{align*}
\end{lemma}
\begin{proof}
Since $\sqrt{|g^{*}|}=\psi^{\frac{2d}{d-4}}\sqrt{|g_{0}|}$ and $(g^{*})^{-1}\sqrt{|g^{*}|}=\psi^{\frac{2(d-2)}{d-4}}\sqrt{|g_0|}g_{0}^{-1}$, we have
\begin{align*}
&\Delta_{g^{*}} u=\frac{1}{\sqrt{|g^{*}|}}\partial_{i}\left(\sqrt{|g^{*}|}g^{ij}\partial_{j} u\right)\\
&=\frac{2(d-2)}{d-4}\psi^{-\frac{d}{d-4}}\langle\nabla\psi,\nabla  u\rangle_{g_{0}}+\psi^{-\frac{4}{d-4}}\Delta_{g_{0}} u
\end{align*}
and since $\psi$ is radial we obtain
\begin{align*}
\Delta_{g^{*}} u=\frac{2(d-2)}{d-4}\psi^{-\frac{d}{d-4}}\psi^{'}\frac{\partial  u}{\partial r}+\psi^{-\frac{4}{d-4}}\Delta_{g_{0}} u,
\end{align*}
and since the normal direction $n_{g_{0}}$ coincides with the radial direction $\frac{\partial}{\partial r}$ we have 
\begin{align*}
\frac{\partial}{\partial n_{g^{*}}}\Delta_{g^{*}} u=\psi^{-\frac{2}{d-4}}\frac{\partial}{\partial r}\left(\frac{2(d-2)}{d-4}\psi^{-\frac{d}{d-4}}\psi^{'}\frac{\partial  u}{\partial r}+\psi^{-\frac{4}{d-4}}\Delta_{g_{0}} u\right).
\end{align*}
Since $ u$ satisfies a Neumann boundary condition we have
\begin{align*}
\frac{\partial}{\partial n_{g^{*}}}\Delta_{g^{*}} u|_{r=1}=\left(\frac{2(d-2)}{d-4}\psi^{'}\frac{\partial^{2}  u}{\partial r^{2}}-\frac{4}{d-4}\psi^{'}\Delta_{g_{0}} u+\frac{\partial}{\partial r}\Delta_{g_{0}} u\right)|_{r=1},
\end{align*}
and writing 
\begin{align*}
\Delta_{g_{0}} u=\frac{\partial^{2}}{\partial r^{2}} u+\frac{d-1}{r}\frac{\partial}{\partial r} u+r^{-2}\tilde{\Delta} u,
\end{align*}
we obtain
\begin{align*}
\Delta_{g_{0}} u|_{r=1}=\frac{\partial^{2}}{\partial r^{2}} u|_{r=1}+\tilde{\Delta} u|_{r=1},
\end{align*} 
and 
\begin{align*}
\frac{\partial}{\partial n_{g^{*}}}\Delta_{g^{*}} u&=2\psi^{'}|_{r=1}\frac{\partial^{2}}{\partial r^{2}} u|_{r=1}+\psi^{'}|_{r=1}\tilde{\Delta} u+\frac{\partial}{\partial r}\Delta_{g_{0}} u|_{r=1}\\
&=-(d-4)\frac{\partial^{2}}{\partial r^{2}} u|_{r=1}+2\tilde{\Delta} u|_{r=1}+\frac{\partial}{\partial r}\Delta_{g_{0}} u|_{r=1}.
\end{align*}
Observe on the other hand that 
\begin{align}
\Delta_{g_{0}}(\psi  u)= u\Delta_{g_{0}}\psi+2\langle\nabla  u,\nabla\psi\rangle_{g_{0}}+\psi\Delta_{g_{0}} u,
\end{align}
and 
\begin{align*}
2\langle\nabla  u,\nabla\psi\rangle_{g_{0}}=2\psi^{'}\frac{\partial}{\partial r} u.
\end{align*}
Since $\Delta_{g_{0}}\psi$ is a constant and $ u$ satisfies a Neumann boundary condition we obtain
\begin{align}
\partial_{n_{g_{0}}}\left(\Delta_{g_{0}}(\psi  u)\right) u&=\frac{\partial}{\partial r}\left(2\psi^{'}\frac{\partial}{\partial r} u+\psi\Delta_{g_{0}} u\right)|_{r=1},\\
&=3\psi^{'}\frac{\partial^{2}  u}{\partial r^{2}}|_{r=1}+\psi^{'}\tilde{\Delta} u|_{r=1}+\frac{\partial}{\partial r}\Delta_{g_{0}} u|_{r=1}\\
&=-\frac{3}{2}(d-4)\frac{\partial^{2}  u}{\partial r^{2}}|_{r=1}-\frac{d-4}{2}\tilde{\Delta} u|_{r=1}+\frac{\partial}{\partial r}\Delta_{g_{0}} u|_{r=1}.
\end{align}
We have obtained 
\begin{align*}
&\oint_{S^{d-1}}\frac{\partial}{\partial n_{g^{*}}}(\Delta_{g^{*}} u)d\sigma-\oint_{S^{d-1}}\frac{\partial}{\partial n_{g_0}}(\Delta_{g_0}\psi  u)d\sigma\\
&=\oint_{S^{d-1}} u\left(-(d-4)\frac{\partial^{2}}{\partial r^{2}} u|_{r=1}+2\tilde{\Delta} u|_{r=1}+\frac{\partial}{\partial r}\Delta_{g_{0}} u|_{r=1}\right)d\sigma\\
&-\oint_{S^{d-1}} u\left(-\frac{3}{2}(d-4)\frac{\partial^{2}  u}{\partial r^{2}}|_{r=1}-\frac{d-4}{2}\tilde{\Delta} u|_{r=1}+\frac{\partial}{\partial r}\Delta_{g_{0}} u|_{r=1}\right)d\sigma\\
=&\frac{d}{2}\oint_{S^{d-1}}f\tilde{\Delta} fd\sigma+\frac{d-4}{2}\oint_{S^{d-1}} u\frac{\partial^{2}}{\partial r^{2}} u|_{r=1}d\sigma\\
=&\frac{d}{2}\oint_{S^{d-1}}f\tilde{\Delta}fd\sigma+\frac{d-4}{2}\oint_{S^{d-1}}f\frac{\partial^{2}}{\partial n_{g_{0}}^{2}} ud\sigma.
\end{align*}
 \end{proof}
We have shown 
\begin{corollary}\label{I1_I2}
For $u$ satisfying \eqref{udir} and \eqref{uneumann}, the following identity holds 
\begin{align*}
\II_{1}(g^{*})[ u]-\II_{2}(g^{*})[ u]&-\left(\II_{1}(g_{0})[\psi  u]-\II_{2}(g_{0})[\psi  u]\right) \\
&=-2\oint_{S^{d-1}}f\tilde{\Delta}f d\sigma-\frac{d(d-4)^2}{4}\oint_{S^{d-1}}f^{2}d\sigma.
\end{align*}
 \end{corollary}
\begin{corollary}\label{cor_energy}
The energy identity \eqref{eqn e4} is equivalent to
\begin{align}\label{alternative_energy}
\II_{1}(g_{0})[\psi U_{f}]-\II_{2}(g_{0})[\psi U_{f}]&=2\oint_{S^{d-1}}fP_{3}fd\sigma-(d-4)\oint_{S^{d-1}}Q_{3}f^{2}d\sigma\\
                                                     &-2\oint_{S^{d-1}}|\tilde{\nabla}f|^2d\sigma+\frac{d(d-4)^2}{4}\oint_{S^{d-1}}f^{2}d\sigma,
                                                     \end{align}
where $Q_{3}$ is given by \eqref{eqn:q_scattering}.
\end{corollary}
\section{Proof of Theorem \ref{dge4}}\label{proofdge4}
We first assert that there is a class of pseudo differential operators of fractional order $ 2 \gamma \leq d-1$  that are intrinsically defined on
$S^{d-1}$ with the conformal covariant property.  The formula was explicitly computed by Branson \cite{branson95}. In the special case when $2 \gamma =3$, we denote 
the operator as $\PP_{3}$, the formula of $\PP_{3}$ is given as:        
\begin{align}\label{defPP3}
 \PP_{3}=(B-1)B(B+1),
\end{align}
where $B=\sqrt{-\tilde{\Delta}+\left(\frac{d-2}{2}\right)^2}$. Note that $B$ (and therefore $\PP_{3}$) is completely determined on spherical harmonics by 
\begin{align*}
B\xi^{k}=\left(k+\frac{d-2}{2}\right)\xi^k, 
\end{align*}
where $\xi^{k}$ is a spherical harmonic of order $k$. Observe also that for the rough laplacian $\tilde{\Delta}$ with respect to the round metric $g_{S^{d-1}}$ we have 
\begin{align*}
\tilde{\Delta}\xi^{k}=-k(k+d-2)\xi^{k}.
\end{align*} 
The operator $\PP_{3}$ plays a role in the formulation of the following sharp Moser-Trudinger and Sobolev inequalities proved in \cite{be93}. Before stating the results in \cite{be93}, 
we adopt a convention about normalized measures intended to simplify the proof of Lemma \ref{lemdN} and Theorem \ref{dge4}. 
\begin{convention}\em{
We will use the following convention for normalized measures:
\begin{itemize}
\item On $S^{d-1}$, $d\sigma$ will be the unnormalized Lebesgue 
measure with respect to the spherical metric $g_{S^{d-1}}$. We will 
use $d\xi$ to denote the normalized measure 
\begin{align}\label{xinorm}
d\xi=\frac{dV_{g_{S^{d-1}}}}{|S^{d-1}|}=\frac{d\sigma}{\frac{2\pi^{d/2}}{\Gamma(d/2)}},
\end{align}
where $|S^{d-1}|$ is the volume of $S^{d-1}$ with respect to the metric $g_{S^{d-1}}$. 
\item $dx$ will denote the Lebesgue measure on $B^{d}$ with 
respect to the Euclidean metric and $d\mu$ will be the measure
\begin{align}\label{munorm}
d\mu=\frac{dx}{|B^{d}|},
\end{align}
where $|B^{d}|$ is the volume of the $d$-ball $B^{d}$ with respect to the Euclidean metric, note that $d |B^{d}| = |S^{d-1}|$. 
\end{itemize}
}
\end{convention}
We observe that from the definition of the normalized measures $d\xi$ and $d\mu$ in \eqref{xinorm} and \eqref{munorm} respectively, the measure $|S^{d-1}|^{-1}dV_{g_{0}}=\frac{dx^2}{|S^{d-1}|}$ is such that
\begin{align*}
\left(\left|S^{d-1}\right|^{-1}dV_{g_{0}}\right)\big{|}_{S^{d-1}}=d\xi|_{\partial B^4},
\end{align*}
and also
\begin{align}\label{normalizeddx}
\frac{dV_{g_{0}}}{|S^{d-1}|}=\frac{|B^{d}|}{|S^{d-1}|}d\mu=\frac{1}{d}d\mu.  
\end{align}
We now state the main inequalities in \cite{be93} that we are interested in.   
\begin{theorem}[\cite{be93}]\label{thmP3} Let $f\in C^{\infty}(S^{d-1})$ and let $\PP_{3}$ be the operator defined on $S^{d-1}$ by \eqref{defPP3}. Then 
\begin{enumerate}
\item[(a)] If $d=4$ and $\varphi\in C^{\infty}(S^{3})$ then
\begin{align*}
\log\left(\oint_{S^{3}} e^{\varphi-\bar{\varphi}}d\xi\right)\le\frac{1}{2(3)!}\oint_{S^{3}}\varphi\PP_{3}\varphi d\xi.
\end{align*}
\item[(b)] If $d>4$ and $f\in C^{\infty}(S^{d-1})$ then
\begin{align*}
a_{d} \left(\oint_{S^{d-1}}|f|^{\frac{2(d-1)}{(d-4)}}d\xi\right)^{\frac{2(d-1)}{(d-4)}}\le \oint_{S^{d-1}}f\PP_{3}f d\xi,
\end{align*}
where $a_{d}=\frac{d(d-2)(d-4)}{8}$. 
\end{enumerate}
\end{theorem}
The operator $\PP_{3}$ above stated is closely related to the scattering operator $P_{3}$ mentioned in Proposition \ref{properties}, in fact, Branson proved the following
\begin{theorem}[\cite{bra93}]\label{bransonthm}
On the model space $(B^d,S^{d-1},g_{0})$ for $d>4$ and for any $g$ in the same conformal class of $g_{0}$, we have 
\begin{enumerate} 
 \item $P_{3}=\PP_{3}$.
 \item $Q_{3}(g)= \frac {2}{d-4} \PP_{3}(1)=\frac{d(d-2)}{4}$. 
\end{enumerate}
 \end{theorem}
We will now start the proof of Theorem A by first justifying the boundary condition \eqref{dNeumann} imposed in the statement of the theorem.
\begin{lemma} \label{lemdN}
A function $u$ satifies \eqref{udir} and \eqref{uneumann} if and only if the function $v = \psi u$ is in $\ff_{f}$, where $\ff_{f}$ denotes 
the class  of functions
\begin{align*}
 \ff_{f} =\left\{ v \in C^\infty (B^d):v =f~\text{on}~S^{d-1},~\text{and}~ \frac{ \partial v}{\partial n_{g_0} } = - \frac {d-4}{2} f ~\text{on}~ S^{d-1}\right\}. 
\end{align*}
\end{lemma}
\begin{proof}[Proof of Lemma \ref{lemdN}] This follows directly from the properties of the function $\psi$, which equals one 
 and whose Neumann derivative equals $ - \frac{d-4}{2} $ on the boundary $S^{d-1}$ of 
the Euclidean ball $B^d$.
\end{proof}
\vskip .1in
\begin{proof}[Proof of Theorem \ref{dge4}]  We start by noting that since $Q_{3}(g^{*})=\frac{d(d-2)}{4}$, we can rewrite the energy identity \eqref{alternative_energy} as
\begin{align*}
\oint_{S^{d-1}}fP_{3}fd\sigma=&\frac{1}{2}\II_{1}(g_{0})[\psi U_{f}]-\frac{1}{2}\II_{2}(g_{0})[\psi U_{f}]+\oint_{S^{d-1}}|\tilde{\nabla} f|^2d\sigma\\
&+\frac{d(d-4)}{4}\oint_{S^{d-1}}|f|^{2}d\sigma,
\end{align*} 
and by Theorem \ref{bransonthm} we see that
\begin{align}\label{P3identity}
 \oint_{S^{d-1}}fP_{3}fd\sigma=\oint_{S^{d-1}}f\PP_{3}f d\sigma.
\end{align}
On the other hand, if we set $v_f=\psi U_{f}$ we see that $v_f$ is a biharmonic extension of $f$ to $B^{4}$ and $v_f$ is in $\ff_{f}$,  
Since $v_f$ is bi-harmonic, we have
\begin{align*}
\frac{1}{2}\II_{1}(g_{0})[\psi U_{f}]-\frac{1}{2}\II_{2}(g_{0})[\psi U_{f}]&=\frac{1}{2}\II_{1}(g_{0})[v_f]-\frac{1}{2}\II_{2}(g_{0})[v_f]\\
&=\frac{1}{2}\int_{B^{4}}(\Delta_{g_{0}}v_f)^{2}dx.
\end{align*}
from Theorem \ref{thmP3} part (b), \eqref{P3identity} and \eqref{normalizeddx} we obtain 
\begin{align}
a_{d}\left(\oint_{S^{d-1}}|f|^{\frac{2(d-1)}{d-4}}d\xi\right)^{\frac{d-4}{2(d-1)}}&\le\oint_{S^{d-1}}fP_{3}fd\xi\nonumber\\
&=\frac{1}{2d}\int_{B^{4}}(\Delta_{g_{0}} v_f)^{2}d\mu+\oint_{S^{d-1}}|\tilde{\nabla} f|^2 d\xi \nonumber\\
&+\frac{d(d-4)}{4}\oint_{S^{d-1}}f^2d\xi.\label{ineA}
 \end{align}
Thus for a general function $v \in \ff_{f}$, 
it is clear that $\displaystyle{\int_{B^{4}}(\Delta_{g_{0}}v)^{2}d\mu}$ is minimized by bi-harmonic functions 
in $\ff_{f}$ which yields the desired inequality. The assertion on the functions for which equality 
in \ref{mainineqge4} is attained follows from an observation made in \cite{be93} where it is proved that equality in Theorem \ref{thmP3} 
part (b) is attained for functions of the form 
$f(\xi)=c|1-\langle z_{0},\xi\rangle|^{\frac{4-d}{4}}$ where $c$ is a constant. We remark that for $v \in \ff_{f}$, inequality \eqref{ineA} is
the same as \eqref{mainineqge4} in the statement of Theorem \ref{dge4} after changing the notation $d \xi$, $d\mu$ back to $d\sigma$, $dx$. 
\end{proof}

\section{Proof of Theorem \ref{deq4}}\label{proofeq4}
 Recall that  when $d=4$, the metric $g^*$ is given by $g^*=e^{2\rho}g_{0}$ where $\rho(x)=\frac{1-|x|^2}{2}$, and therefore, for the Paneitz operator $P_4$ we have 
 \begin{align}\label{P44conf}
 (P_{4})_{g^*}=e^{-4\rho}(\Delta_{g_{0}})^2. 
 \end{align}
Fixing a function $\varphi\in C^{\infty}(S^3)$, recall that we have defined the class $C_{\varphi}$ given by 
\begin{align*}
C_{\varphi}=\left\{w\in C^{\infty}(B^4):w|_{S^3}=\varphi,\frac{\partial}{\partial n}w|_{S^3}=0\right\}.
\end{align*}  
 
Assuming that $w_{\varphi}$ is a bi-harmonic function in $C_{\varphi}$, we first recall that the energy identity \eqref{eqn e4}  in the case $d=4$ is of the form
\begin{align}\label{energyd=4}
E_{4}(g^*)[w_{\varphi}]=2\oint_{S^3}\varphi P_{3}\varphi d\sigma.
\end{align}
Using the identities \eqref{E1} and \eqref{E2} together with the proof of Lemma \ref{I3lemma} we reach the conclusion that \eqref{energyd=4} is equivalent to
\begin{align}
-\II_{2}(g^*)[w_{f}]=2\oint_{S^3}\varphi P_{3}\varphi d\sigma.\label{new4energy}
\end{align}
We now compute the difference $\II_{2}(g_{*})[w_{\varphi}]-\II_{2}(g_{0})[w_{\varphi}]$. We start by noting that for any $w$ in $C_{\varphi}$ we have 
\begin{align*}
\frac{\partial}{\partial n_{g^{*}}}(\Delta_{g^{*}}w)|_{S^{3}}&=\frac{\partial}{\partial n_{g_{0}}}(\Delta_{g_{0}}w)|_{S^{3}}+2\Delta_{g_{0}}w-2\frac{\partial^2}{\partial n^2_{g_{0}}}w|_{S^{3}}\\
 &=\frac{\partial}{\partial n_{g_{0}}}(\Delta_{g_{0}}w)|_{S^{3}}+2\tilde{\Delta}\varphi,
\end{align*}
thus \eqref{new4energy} takes the form
\begin{align}
-\II_{2}(g_{0})[w_{\varphi}]-2\oint_{S^{3}}\varphi\tilde{\Delta}\varphi d\sigma=2\oint_{S^{3}}\varphi P_{3}\varphi d\sigma.                                          
\end{align}
Integrating by parts we obtain 
\begin{align}
\int_{B^{4}}(\Delta_{g_{0}} w_{\varphi})^{2}dV_{g_{0}}+2\oint_{S^{d-1}}|\tilde{\nabla}\varphi|^{2}d\sigma=2\oint_{S^{d-1}}\varphi P_{3}\varphi d\sigma. 
\end{align}
We now prove Theorem \ref{deq4}.
\begin{proof}[Proof of Theorem \ref{deq4}]
Normalizing the measure $d\sigma$ to obtain $d\xi$ and assuming that $\bar{\varphi}=0$ we see from Theorem \ref{thmP3} part (a) that 
\begin{align*}
\log\left(\oint_{S^3}e^{3\varphi}d\xi\right)\le\frac{9}{2(3)!}\oint_{S^{3}}\varphi P_{3}\varphi d\xi=\frac{1}{4(4!)}\int_{B^{4}}(\Delta_{g_{0}} w_{\varphi})^2d\mu+\frac{9}{2(3)!}\oint_{S^{3}}|\tilde{\nabla} \varphi|^2d\xi, 
\end{align*}
and clearly $w_{\varphi}$ minimizes $\int_{B^{4}}(\Delta_{g_{0}}w_{\varphi})^2d\mu$ in the class $C_{\varphi}$, and this yields the desired inequality. As in the proof of Theorem \ref{dge4}, the assertion about the functions for which equality is attained in \eqref{mainineq4} follows from 
the form of the functions for which equality is achieved in Theorem \ref{thmP3} part (a). As proved by Beckner in \cite{be93}, all of these functions are of the form 
$\varphi=-\log\left|1-\langle z_{0},\xi\rangle\right|+c$ where $c$ is a constant, $\xi\in S^{3}$ and $|z_{0}|<1$ is fixed.  
 \end{proof}
\section{Proof of Theorem \ref{thmFD}}\label{pfthmFD}
We start by describing the term ${\bf I}_{2}[g,w]$ appearing in \eqref{11terms} explicitly, however,  before writing out the formula proved in \cite{cqI}, we need to describe important boundary analogues of the Paneitz operator and of $Q$-curvature for compact manifolds with boundary in dimension $4$.
Consider a manifold with boundary $(M^{4},\partial M,g)$ and let $\tilde{g}$ be the restriction of $g$ to $\partial M$. Let $n$ be the unit \emph{outward} normal and let us use  greek letters to denote directions tangent to the boundary, i.e., $g_{\alpha\beta}=\tilde{g}_{\alpha\beta}$. Observe that with this convention the second fundamental form is given by 
\begin{align}
\LL_{\alpha\beta}=\frac{1}{2}\frac{\partial}{\partial n}g_{\alpha\beta},
\end{align} 
and let us use $H$  to denote the mean curvature, i.e., $H=g^{\alpha \beta}\LL_{\alpha \beta}$. The boundary operator $P^{b}_{3}$ defined in \cite{cqI} is
\begin{align}\label{P3b}
P^{b}_{3}w=-\frac{1}{2}\frac{\partial}{\partial n}\Delta_{g}w-\tilde{\Delta}\frac{\partial}{\partial n}w-\frac{2}{3}H\tilde{\Delta}w+\LL^{\alpha\beta}\tilde{\nabla}_{\alpha}\tilde{\nabla}
_{\beta}w+\frac{1}{3}\tilde{\nabla}H\cdot\tilde{\nabla}w+(F-2J)\frac{\partial}{\partial n}w,
\end{align}
where $\tilde{\Delta},\tilde{\nabla}$ are the boundary counterparts of $\Delta_{g},\nabla$. Note that this time the function $w$ is defined in all of $\overline{B}^{4}$. The convention that we use for the Laplacian is that  $\Delta_{g}w$ stands for
\begin{align*}
\Delta_{g}w=\nabla_{n}\nabla_{n}w+g^{\alpha\beta}\nabla_{\alpha}\nabla_{\beta}w,
\end{align*} 
where $\nabla_{n}\nabla_{n}w$ means two covariant derivatives with respect to the metric $g$ in the normal direction $n$. The terms $F$ and $J$ and \eqref{P3b} are used to denote
\begin{itemize}
\item $F=\ric_{nn}$,
\item $J=\frac{R_{g}}{6}|_{\partial M}$, i.e., the restriction to the boundary of the ambient scalar curvature $R_{g}$.    
\end{itemize}  
We also have a boundary curvature $T_{3}$ given by 
\begin{align}\label{T3}
T_{3}=-\frac{1}{2}\frac{\partial}{\partial n}J+JH-\langle G,\LL\rangle +\frac{1}{9}H^{3}-\frac{1}{3}\tr_{\tilde{g}}(\LL^{3})+\frac{1}{3}\tilde{\Delta}H
 \end{align}
 where 
 \begin{itemize}
 \item $G_{\beta}^{\alpha}=R^{\alpha}_{n\beta n}$,
 \item $\langle G,\LL\rangle=R_{\alpha n\beta n}\LL_{\alpha\beta}$,
 \item $\LL^{3}_{\alpha\beta}=\LL_{\alpha\gamma}\LL^{\gamma\delta}\LL_{\delta\beta}$.
 \end{itemize}
For the model case $(B^{4},S^{3},ds^{2}_{\R^{4}})$ we easily see that 
 \begin{align}
 \LL&=g_{S^{3}},\\
 H&=3,
 \end{align}
and  $\langle G,\LL\rangle=F=J=0$ so that if $n$ is the outer normal to $S^{3}$ with respect to $g_{0}=ds^{2}$
\begin{align}\label{P3bflat}
P_{3}^{b}=-\frac{1}{2}\frac{\partial}{\partial n}\Delta_{g_{0}}-\tilde{\Delta}\frac{\partial}{\partial n}-\tilde{\Delta}
\end{align}
and 
\begin{align}
T_{3}=\frac{1}{9}H^{3}-\frac{1}{3}\tr(\LL^{3})=2. 
\end{align}
The formula proved by Chang and Qing in \cite{cqI} is the following

\begin{align}\label{I2formula}
{\mathbf I}_{2} [g, w]&=b_{2}[g_{0},w]+\frac{1}{12}D[g_{0},w],
 \end{align}
 where 
 \begin{align}\label{b2}
 b_{2}[g,w]=\frac{1}{4}\int_{B^{4}}w(P_{4})_{g}wdV_{g}+\frac{1}{2}\int_{B^{4}}wQ_{g}dV_{g}+\frac{1}{2}\int_{S^{3}}wP^{b}_{3}wd\sigma+\int_{S^{3}}wT_{3}d\sigma,
 \end{align}
 and \footnote{Note that we are using a different sign convention than the one used by Chang and Qing in \cite{cqI,cqII} since they consider the inner unit normal and not the outer normal.}
 \begin{align}
 D[g,w]=-\oint_{S^3}(-R_g-\frac{1}{3}H^2+3F)\frac{\partial}{\partial n}wd\sigma-\oint_{S^3}H\frac{\partial^2}{\partial n^2}wd\sigma-\oint_{S^3}H\tilde{\Delta}wd\sigma.
 \end{align}
We have used in \eqref{b2} the same notation employed in \cite{cqI}. For the model case $(B^4,S^3,g_0)$, a simple computation shows that 
\begin{align}\label{Dmodel}
D[g_{0},w]=3\oint_{S^3}\left(\frac{\partial}{\partial n}w-\frac{\partial^2}{\partial n^2}w\right).
\end{align} 
 
 It is well known that the Paneitz operator $(P_{4})_{g}$ and Q-curvature $(Q_4)_{g}$ have conformal covariance properties, but it turns out 
that $P^{b}_{3}$ and $T_{3}$ are also conformally covariant. More precisely, letting $g_{\tau}=e^{2\tau}g$,we have the transformation laws
\begin{enumerate}
\item ${(P^{b}_{3})}_{g _\tau} = e^{-3\tau} {(P^{b}_{3})}_{g}$,
\item ${(P^{b}_{3})}_{g} \tau+ (T_{3})_{g} = (T_3)_{g_\tau} e^{3\tau}$.
\end{enumerate}
Applying the conformal covariance properties above  to the pair of metrics $g_{0}$ and $g^*=e^{2\rho}g_{0}$, we easily obtain the following
\begin{lemma}\label{PTg*}
The metric $g^*$ has the following properties
\begin{enumerate}
\item $(Q_{4})_{g^*}=(Q_{4})_{g_0}=0$,
\item $(T_{3})_{g^*}=(T_{3})_{g_0}=2$,
\end{enumerate}
hence
\begin{align}\label{invarianceb2}
b_{2}[g^*,w]=b_{2}[g_{0},w],
\end{align}
for all $w\in C^{\infty}(B^4)$. 
\end{lemma}

It was pointed out by Chang and Qing in \cite{cqII} that on the model space $(B^{4},S^{3},|dx|^2)$ where again $g_{0}=|dx|^2$ is the Euclidean metric, there is a close relationship between the operator $P^b_{3}$ and Beckner's $\PP_{3}$ operator. In fact, we have 
\begin{lemma}[\cite{cqII}, Lemma 3.3 and Corollary 3.1]\label{P3lemma} Suppose that $w$ satisfies 
\begin{align}
\Delta_{g_{0}}^{2}w=0~\text{on}~B^{4},\label{bilap}\\
w|_{S^3}=\varphi,\label{wbound}
\end{align}
Then 
\begin{align*}
P_{3}^bw|_{S^3}=\PP_{3}\varphi.
\end{align*}
\end{lemma}
\begin{remark} \em{We point out that Lemma \ref{P3lemma} holds regardless of the value of the normal derivative of $w$ at $S^3$. 
}
\end{remark}
 As before we will be interested in the class $C_{\varphi}$ of functions $w$ defined on $B^4$   satisfying $w|_{S^3}=f$  and $\frac{\partial}{\partial n}w|_{S^3}=0$.  
In addition to conformal invariance,  the functional $\ii_2$ satisfies the following property for the metrics $g_{0}$ and $g^*=e^{2\rho}g_{0}$:
\begin{lemma}\label{I2property}
For any function $w$ in $C_{\varphi}$ and $g^*=e^{2\rho}g_{0}$ where again $\rho=\frac{(1-|x|^2)}{2}$ we have 
\begin{align*}
\ii_{2}[g^*,w]=\ii_{2}[g_{0},w+\rho].
\end{align*}
\end{lemma}
\begin{proof}
Since $(B^4,g^*)$ has a totally  geodesic boundary and $w\in C_{\varphi}$, we  observe that $D[g^*,w]=0$.  It then follows from  Lemma \ref{PTg*} we have
 \begin{align}\label{I2b2}
\ii_{2}[g^*,w]=b_{2}[g^*,w]=b_{2}[g_{0},w]. 
\end{align}
From the expression for $D[g_{0},w]$ in \eqref{Dmodel} we have 
\begin{align}
\ii_{2}[g_{0},w+\rho]&=b_{2}[g_{0},w+\rho]+\frac{1}{12}D[g_{0},w]\nonumber\\
                                   &=b_{2}[g_{0},w+\rho]+\frac{1}{4}\oint_{S^3}(w+\rho)_{n}d\sigma-\frac{1}{4}\oint_{S^3}(w+\rho)_{nn}d\sigma.\label{I2lastline}
\end{align}
We now compute all terms involved in the expression of $b_{2}[g_{0},w+\rho]$. The function $\rho$ satisfies the following equations with respect to the flat metric
\begin{align*}
(P_{4})_{g_{0}}\rho&=0,\\
P^{b}_{3}\rho&=0,
\end{align*}
and from these identities we observe that
\begin{align}
\int_{B^4}(w+\rho)(P_{4})_{g_{0}}(w+\rho)dV_{g_{0}}&=\int_{B^4}w(P_{4})_{g_{0}}w dV_{g_{0}}+\int_{B^4}\rho (P_{4})_{g_{0}}w dV_{g_{0}},\label{4rho}\\
\oint_{S^3}(w+\rho)(P^{b}_{3})_{g_{0}}(w+\rho)d\sigma&=\oint_{S^3}w(P^b_{3})_{g_{0}}w d\sigma.\label{3rho}
\end{align}
Our goal now is to compute the term $\int_{B^4}\rho (P_{4})_{g_{0}}w dV_{g_{0}}$. Integrating by parts we have 
\begin{align}
\int_{B^4}\rho (P_{4})_{g_{0}}w dV_{g_{0}}&=\int_{B^4}\rho (\Delta^2_{g_{0}})w dV_{g_{0}}\nonumber\\
&=\int_{B^4}\left(\Delta_{g_{0}}\rho\right)\left(\Delta_{g_{0}}w\right)dV_{g_{0}}+\oint_{S^3}\rho\frac{\partial}{\partial n}\left(\Delta_{g_{0}}w\right)d\sigma-\oint_{S^3}\frac{\partial\rho}{\partial n}\Delta_{g_{0}}w d\sigma\nonumber\\
&=-4\int_{B^4}\Delta_{g_{0}}wdV_{g_{0}}+\oint_{S^3}\Delta_{g_{0}}wd\sigma,\nonumber \\
&=\oint_{S^3}-4\frac{\partial w}{\partial n}+\Delta_{g_{0}}w d\sigma.\label{rhoPwlast}
\end{align}
Using \eqref{rhoPwlast} together with the expression $\Delta_{g_{0}}|_{S^3}=w_{nn}+3w_{n}+\tilde{\Delta}w$ we conclude that
\begin{align}\label{rhoPw}
\int_{B^4}w\left(P_{4}\right)_{g_0}w dV_{g_{0}}=\oint_{S^3}(w_{nn}-w_{n})d\sigma.
\end{align} 
Combining \eqref{4rho},\eqref{3rho}, Lemma \ref{PTg*} and \eqref{rhoPw} we obtain
\begin{align}\label{b2rho}
b_{2}[g_{0},w+\rho]=b_{2}[g_{0},w]+\frac{1}{4}\oint_{S^3}(w_{n}-w_{nn})d\sigma,
\end{align}
and combining \eqref{b2rho} with  \eqref{I2b2} and \eqref{I2lastline} we conclude that
\begin{align*}
\ii_{2}[g^*,w]=\ii_{2}[g_{0},w+\rho],
\end{align*}
for all $w\in C_{\varphi}$ as needed.
\end{proof}

We now prove Theorem \ref{thmFD}: 

\begin{proof}[Proof of Theorem of \ref{thmFD}] We start by proving part (1). Since $w\in C_{\varphi}$, we have from equation \eqref{invarianceb2} 
\begin{align*}
\ii_{2}[g^*,w]=b_{2}[g_{0},w].
\end{align*}
Recall that for the model space $(B^4,S^3,g_{0})$ and for $w\in C_{\varphi}$ we have 
\begin{align}
(P_{4})_{g_0} w &=\Delta^2_{g_{0}}w,\\
(P^b_{3})_{g_{0}} w &=-\frac{1}{2}\frac{\partial}{\partial n}\Delta_{g_{0}} w -\tilde{\Delta}\varphi,\\
(T_{3})_{g_0}&=2,\\
(Q_4)_{g_0}&=0,
\end{align}
and a simple computation shows that
\begin{align*}
\ii_{2}[g^*,w]=b_{2}[g_{0},w]=\frac{1}{4}\int_{B^4}(\Delta_{g_0} w)^2dV_{g_{0}}+\frac{1}{2}\oint_{S^3}|\tilde{\nabla}\varphi|^2d\sigma+2\oint_{S^3}\varphi d\sigma. 
\end{align*}
From Theorem \ref{deq4} we have the lower bound
\begin{align*}
\ii_{2}[g^*,w]&=\frac{1}{4}\int_{B^4}(\Delta_{g_0} w)^2dV_{g_{0}}+\frac{1}{2}\oint_{S^3}|\tilde{\nabla}\varphi|^2d\sigma+2\oint_{S^3}\varphi d\sigma\\
&\ge\frac{4\pi^2}{3}\log\left(\frac{1}{2\pi^2}\oint_{S^3}e^{3(\varphi-\bar{\varphi})}d\sigma\right)+2\oint_{S^3}\varphi d\sigma,
\end{align*}
and since $\displaystyle{\oint_{S^3}e^{f}d\sigma=|S^3|=2\pi^2}$ we obtain
\begin{align*}
\ii_{2}[g^*,w]\ge -4\pi^2\bar{\varphi}+2\oint_{S^3}\varphi d\sigma=-2\oint_{S^3}\varphi d\sigma+2\oint_{S^3}\varphi d\sigma=0.  
\end{align*}
Now, if $\ii_{2}[g^*,w]=0$ we have 
\begin{align*}
\frac{1}{4}\int_{B^4}(\Delta_{g_0} w)^2dV_{g_{0}}+\frac{1}{2}\oint_{S^3}|\tilde{\nabla}\varphi|^2d\sigma=\frac{4\pi^2}{3}\log\left(\frac{1}{2\pi^2}\oint_{S^3}e^{3(\varphi-\bar{\varphi})}d\sigma\right),
\end{align*}
and from Theorem \ref{deq4} it follows that $\varphi(\xi)=-\log\left|1-\langle z_{0},\xi\rangle\right|+c$ where $z_{0}\in B^4$ is fixed, $\xi\in S^3$, $c$ is a constant and $w$ is a bi-harmonic extension of $\varphi$ in $C_{\varphi}$.\\ 

\indent Part (2) follows from combining part (1) with Lemma \ref{I2property}.
\end{proof}

\bibliography{fractional} 
\end{document}